\documentclass{amsart}
\usepackage[toc,page]{appendix}
\usepackage[headheight=12.1pt]{geometry}
\usepackage[all,cmtip]{xy}
\usepackage{amsmath}
\usepackage{amsfonts}
\usepackage{amssymb}
\usepackage{mathrsfs}   
\usepackage{amsthm}

\usepackage{xcolor}
\usepackage{cite}
\usepackage{stmaryrd}
\usepackage{graphicx}
\usepackage{pgf}
\usepackage{eepic}
\usepackage{pstricks}
\usepackage{epsfig}
\usepackage{fancyhdr}
\usepackage{tikz,caption,float}
\usepackage[backref=page]{hyperref}

\usepackage{euscript}

\usepackage{amsbsy}
\SetSymbolFont{stmry}{bold}{U}{stmry}{m}{n}
\usepackage{bm}

\usepackage{relsize}
\usepackage[bbgreekl]{mathbbol}
\DeclareSymbolFontAlphabet{\mathbb}{AMSb} 
\DeclareSymbolFontAlphabet{\mathbbl}{bbold}

\hypersetup{
     colorlinks   = true,
     citecolor    = red,
     linkcolor = blue,
     urlcolor = purple
}
    
  \newcommand\imCMsym[4][\mathord]{%
  \DeclareFontFamily{U} {#2}{}
  \DeclareFontShape{U}{#2}{m}{n}{
    <-6> #25
    <6-7> #26
    <7-8> #27
    <8-9> #28
    <9-10> #29
    <10-12> #210
    <12-> #212}{}
  \DeclareSymbolFont{CM#2} {U} {#2}{m}{n}
  \DeclareMathSymbol{#4}{#1}{CM#2}{#3}
}
\newcommand\alsoimCMsym[4][\mathord]{\DeclareMathSymbol{#4}{#1}{CM#2}{#3}}

\imCMsym{cmmi}{124}{\CMjmath}
\imCMsym[\mathop]{cmsy}{113}{\CMamalg}
\imCMsym[\mathop]{cmex}{96}{\CMcoprod}
\alsoimCMsym[\mathop]{cmex}{97}{\CMbigcoprod}
 
\numberwithin{equation}{subsection}

\let\realequation\equation
\def\equation{\setcounter{equation}{\arabic{subsubsection}}%
   \refstepcounter{subsubsection}%
   \realequation}
   
\makeatletter
\newcommand*{\relrelbarsep}{.386ex}
\newcommand*{\relrelbar}{%
  \mathrel{%
    \mathpalette\@relrelbar\relrelbarsep
  }%
}
\newcommand*{\@relrelbar}[2]{%
  \raise#2\hbox to 0pt{$\m@th#1\relbar$\hss}%
  \lower#2\hbox{$\m@th#1\relbar$}%
}
\providecommand*{\rightrightarrowsfill@}{%
  \arrowfill@\relrelbar\relrelbar\rightrightarrows
}
\providecommand*{\leftleftarrowsfill@}{%
  \arrowfill@\leftleftarrows\relrelbar\relrelbar
}
\providecommand*{\xrightrightarrows}[2][]{%
  \ext@arrow 0359\rightrightarrowsfill@{#1}{#2}%
}
\providecommand*{\xleftleftarrows}[2][]{%
  \ext@arrow 3095\leftleftarrowsfill@{#1}{#2}%
}
\makeatother

\usepackage{xr,mathtools}

\theoremstyle{plain}
\newtheorem*{theoremu}{Theorem}
\newtheorem{theorem}{Theorem}[section]

\newtheorem{proposition}[theorem]{Proposition}

\newtheorem{corollary}[theorem]{Corollary}

\newtheorem*{corollaryu}{Corollary}
\newtheorem{lemma}[theorem]{Lemma}

\theoremstyle{definition}

\newtheorem{definition}[theorem]{Definition}

\theoremstyle{remark}
\newtheorem{remark}[theorem]{Remark}

\newtheorem{example}[theorem]{Example}

\newcommand{\Z}{{\mathbb Z}}
\newcommand{\Q}{{\mathbb Q}}
\newcommand{\R}{{\mathbb R}}

\renewcommand{\P}{{\mathbb P}}

\renewcommand{\lim}[1]{\underset{#1}{\mathrm{lim}}\,}

\newcommand{\isomto}{\overset{\cong}{\longrightarrow}}

\newcommand{\tate}[1]{\langle #1 \rangle}

\newcommand{\spa}[1]{\mathrm{Spa}\!\left(#1\right)}
\newcommand{\spf}[1]{\mathrm{Spf}\!\left(#1\right)}

\newcommand{\norm}[1]{\left\vert#1\right\vert}

\newcommand{\isoc}[1]{\mathrm{Isoc}^\dagger(#1)}

\title{A note on effective descent for overconvergent isocrystals}

\author{Christopher Lazda}
       \address{Mathematics Institute\\ Zeeman Building\\University of Warwick\\ Coventry \\ CV4 7AL \\ United Kingdom}
       \email{chris.lazda@warwick.ac.uk}

\begin{document}

\begin{abstract} In this short note we explain the proof that proper surjective and faithfully flat maps are morphisms of effective descent for overconvergent isocrystals. We then show how to deduce the folklore theorem that for an arbitrary variety over a perfect field of characteristic $p$, the Frobenius pull-back functor is an equivalence on the overconvergent category.
\end{abstract}

\maketitle 

\tableofcontents

\section*{Introduction}

Let $k$ be a field of characteristic $p>0$. One of the major problems in arithmetic geometry over the last 50 years or so has been that of describing a `good' category of coefficients for $p$-adic cohomology of varieties over $k$, with behaviour mirroring that of the category of $\ell$-adic \'etale sheaves for $\ell\neq p$. The first attempt at doing so was the category of crystals on a variety, introduced by Grothendieck in \cite{Gro68} and more fully developed by Berthelot in his thesis \cite{Ber74}. However, this category fails one of the basic requirements that one expects of such a `good' category of coefficients, namely topological invariance. This manifests itself in the fact that the Frobenius pull-back functor
\[ F^*: \mathrm{Crys}(X/W)_{\Q} \rightarrow  \mathrm{Crys}(X/W)_{\Q} \]
on isocrystals is not necessarily an equivalence of categories, even if $k$ is perfect and $X$ is smooth and proper, for a counter-example see \cite[Proposition 2.13]{ES18}. This problem was rectified by the introduction of the category of convergent isocrystals in \cite{Ogu84}, which when $k$ is perfect and $X$ is smooth turns out to be the largest full sub-category of $\mathrm{Crys}(X/W)_{\Q}$ on which $F^*$ is an equivalence \cite[Proposition 2.3, Remark 2.4]{ES18}. This characterisation is deduced in part from the fact that the category of convergent isocrystals satisfies descent under proper and surjective morphisms of varieties, which in turn implies the required topological invariance. 

When $X$ is not proper, Berthelot introduced in \cite{Ber96b} a refinement of the category of convergent isocrystals on $X$, by considering `overconvergence conditions' (on both objects and morphisms) along the boundary of some compactification $X\hookrightarrow \overline{X}$. A natural question then arises of whether or not this category of overconvergent isocrystals satisfies proper descent, and this was originally proved by Shiho \cite[Proposition 7.3]{Shi07b}. In fact, Shiho works in rather greater generality, that is, he works relative to a general $p$-adic formal scheme of finite type over a DVR, and with partially overconvergent cohomology of pairs. The purpose of this note is to explain the proof of the following version of descent for overconvergent isocrystals, using flat descent in adic geometry.

\begin{theoremu}[\ref{theo: main}] Let $f:X\rightarrow Z$ be a proper surjective morphism of varieties over $k$. Then $f$ is a morphism of effective descent for overconvergent isocrystals.
\end{theoremu}

Note that the `full faithfulness' part of this descent result also follows from the more general results on cohomological descent proved in \cite{Tsu03b,ZB14b}. Recall that the h topology on the category of varieties over $k$ is that for which covering families are given by universal topological epimorphisms \cite[Definition 3.1.1]{Voe96}. Equivalently, it is the coarsest topology for which both Zariski open covers and proper surjective maps are coverings \cite[Theorem 3.1.9]{Voe96}, and so we deduce from Theorem \ref{theo: main} that overconvergent isocrystals satisfy h descent. In particular, since the fppf topology is coarser than the h topology \cite[p.122]{Voe96} it follows that overconvergent isocrystals satisfy fppf decsent. As another, essentially formal, consequence of h descent, we obtain invariance of $\isoc{X/K}$ under universal homeomorphisms, and in particular the following folklore theorem.

\begin{corollaryu}[\ref{cor: Fequiv}] Assume that $k$ is perfect, and let $X$ be a variety over $k$. Then the Frobenius pull-back functor
\[ F^*:\isoc{X/K} \rightarrow \isoc{X/K}\]
is an equivalence of categories.
\end{corollaryu}

For smooth varieties with good compactifications, this follows from Berthelot's theorem on Frobenius descent for arithmetic $\mathcal{D}^\dagger$-modules \cite{Ber00}. A more general version of this result (and therefore the deduction of Corollary \ref{cor: Fequiv}) forms part of current work in progress of Crew (see the introduction to \cite{Cre17} for details). The proof via Theorem \ref{theo: main}, however, is reasonably direct (i.e. does not depend on any results on arithmetic $\mathcal{D}^\dagger$-modules), and Theorem \ref{theo: main} itself is potentially of independent interest.

The key input into the proof of Theorem \ref{theo: main} is the following version of flat descent in analytic geometry.

\begin{theoremu}[\ref{theo: cohdes}] Let $f:X\rightarrow Y$ be a faithfully flat morphism of adic spaces locally of finite type over a complete, discretely valued field. Then $f$ is a morphism of effective descent for coherent sheaves.
\end{theoremu}

This is essentially just a rephrasing of the descent results of \cite[\S4]{Con06}; the key point is the rather satisfying observation, which must be well-known to the experts, that Conrad's condition that a flat map of rigid analytic spaces (in the sense of Tate) `admits local fpqc sections' translates \emph{exactly} into the surjectivity of the associated map on adic spaces. We would therefore like to view this result as yet more evidence (if it were needed) that Huber's theory of adic spaces really is the correct setting in which to do non-archimedean analytic geometry. 

Given this analytic descent result the proof of Theorem \ref{theo: main} proceeds more or less as expected, and is entirely similar to the proof of the corresponding result in \cite{Shi07b}. The point is that a projective, surjective map of varieties can, locally on the base, be extended to a proper flat morphism of frames. One can then show that the induced morphism on suitably small neighbourhoods of the respective tubes is faithfully flat (in the sense of adic geometry) and therefore is a morphism of effective descent for coherent sheaves. Applying this universally in frames mapping to the base $Z$ and using Le Stum's `site-theoretic' interpretation of $\isoc{X/K}$ given in \cite[\S8]{LS07} completes the proof. The faithfully flat case can then easily be deduced.

\subsection*{Acknowledgements} The author was supported by a Marie Curie fellowship of the Istituto Nazionale di Alta Matematica ``F. Severi''. He would like to thank A. Shiho, A. P\'al, B. Chiarellotto, R. Crew, B. Le Stum and N. Mazzari for useful conversations regarding the contents of this note, as well as two anonymous referees for careful readings of the article. Their helpful suggestions have much improved it.

\subsection*{Notations and conventions} We will let $\mathcal{V}$ be a complete, discrete valuation ring with fraction field $K$ of characteristic $0$ and residue field $k$ of characteristic $p>0$. We will let $\varpi$ be a uniformiser of $\mathcal{V}$. A variety over $k$ will be a separated scheme of finite type, and the category of these objects will be denoted $\mathbf{Var}_k$. All formal schemes over $\mathcal{V}$ will be assumed to be of finite type.

An analytic space over $K$ will be an adic space locally of finite type over $\spa{K,\mathcal{V}}$. It will be called an analytic variety over $K$ if in addition the structure morphism to $\spa{K,\mathcal{V}}$ is separated. Similarly, we will refer to a rigid analytic space in the sense of Tate \cite[\S9.3.1]{BGR84} as a `rigid space', and when it is separated over $\mathrm{Sp}(K)$ we will call it a rigid variety. Thus by \cite[\S1.1.11]{Hub96} there is a fully faithful functor from rigid spaces over $K$ to analytic spaces over $K$, which induces an equivalence on the full-subcategories of quasi-separated objects. We will denote the category of analytic spaces over $K$ by $\mathbf{An}_K$. We will abbreviate `quasi-compact and quasi-separated' as `qcqs'.

\section{The formalism of descent}

In this section we will very briefly recall the formalism of descent, and introduce the three examples that particularly interest us, namely coherent sheaves on analytic spaces over $K$, $j^\dagger$-modules on frames over $\mathcal{V}$, and overconvergent isocrystals on algebraic varieties over $k$.

Suppose that we have a fibred category
\[ \mathcal{F} \rightarrow \mathcal{C}\]
over some base category $\mathcal{C}$. For a given $X\in \mathrm{Ob}(\mathcal{C})$ we define the fibre category $\mathcal{F}_X$ to consist of objects (resp. morphisms) of $\mathcal{F}$ mapping to $X$ (resp. $\mathrm{id}_X$). Then, up to choosing a cleavage, we can combine all of these fibre categories into a lax 2-functor
\[ \mathcal{C}^{\mathrm{op}} \rightarrow \mathbf{Cat}.\]
Concretely, such a lax 2-functor amounts to specifying:
\begin{itemize}
\item for every $X\in \mathrm{Ob}(\mathcal{C})$ a category $\mathcal{F}_X$;
\item for every morphism $f:X\rightarrow Y$ a pull-back functor $f^*:\mathcal{F}_Y\rightarrow \mathcal{F}_X$;
\item for every $X\in \mathrm{Ob}(\mathcal{C})$ a natural isomorphism $\varepsilon_X: \mathrm{id}_X^* \cong \mathrm{id}_{\mathcal{F}_X}$;
\item for every $X\in \mathrm{Ob}(\mathcal{C})$ a natural isomorphism $f^*g^*\cong (gf)^*$;
\end{itemize}
subject to certain conditions \cite[\S3]{Vis05}.

In particular, for any morphism $f:X\rightarrow Y$ in $\mathcal{C}$ we have two pull-back functors
\[ \pi_0^*,\pi_1^*:\mathcal{F}_X \rightarrow \mathcal{F}_{X\times_Y X} \]
associated to the two projections $\pi_i:X\times_Y X \rightarrow X$. Similarly, we have three projections
\[ \pi_{01},\pi_{12},\pi_{02}:X\times_Y X\times_Y X \rightarrow X \times_Y X \]
giving rise to corresponding pull-back functors.

\begin{definition} If $E\in \mathcal{F}_X$ then descent data on $E$ relative to $f$ is an isomorphism $\alpha: \pi_0^*E\isomto \pi_1^*E$ such that 
\[ \pi_{02}^*(\alpha)=\pi_{12}^*(\alpha)\circ \pi_{01}^*(\alpha).\]
\end{definition}

The category of objects in $\mathcal{F}_X$ equipped with descent data is denoted $\mathcal{F}_{X\times_Y X\rightrightarrows X}$, pull-back by $f$ induces a functor
\[ f^*: \mathcal{F}_Y\rightarrow \mathcal{F}_{X\times_Y X\rightrightarrows X}. \]

\begin{definition} We say that $f$ is a morphism of descent for $\mathcal{F}$ if the functor 
\[  f^*:\mathcal{F}_Y\rightarrow \mathcal{F}_{X\times_Y X\rightrightarrows X} \]
is fully faithful. We say that $f$ is a morphism of effective descent for $\mathcal{F}$ if $f^*$ is an equivalence of categories.
\end{definition}

This does not depend on the choice of cleavage, and is intrinsic to the fibred category $\mathcal{F}\rightarrow \mathcal{C}$ and the morphism $f:X\rightarrow Y$ in $\mathcal{C}$. The three key examples of fibred categories we will consider in this note are the following.

\begin{example} \begin{enumerate}
\item As in \cite[\S1]{Cre92}, but fixing the ground field $K$, we will view the category of overconvergent isocrystals $\mathbf{Isoc}^\dagger$ as a fibred category
\[  \mathbf{Isoc}^\dagger \rightarrow \mathbf{Var}_k \]
over the category of $k$-varieties. The associated lax 2-functor takes $X\in \mathbf{Var}_k$ to the category $\isoc{X/K}$ of overconvergent isocrystals on $X/K$, and a morphism $f:X\rightarrow Y$ of varieties to the pull-back functor $f^*:\isoc{Y/K}\rightarrow \isoc{X/K}$.
\item Similarly, we may view the category $\mathbf{Coh}$ of coherent sheaves as a fibred category
\[  \mathbf{Coh} \rightarrow \mathbf{An}_K \]
over the category of analytic spaces over $K$.
\item Let $\mathbf{Frame}_\mathcal{V}$ denote the category of frames over $\mathcal{V}$, that is triples $(X,\overline{X},\mathfrak{X})$ consisting of an open immersion of $k$-varieties $X\rightarrow \overline{X}$ and a closed immersion $\overline{X}\rightarrow \mathfrak{X}$ of separated formal $\mathcal{V}$-schemes. Then taking $(X,\overline{X},\mathfrak{X})$ to the category of coherent $j_X^\dagger\mathcal{O}_{]\overline{X}[_\mathfrak{X}}$-modules (in the sense of adic geometry) gives rise to a lax 2-functor, and hence a fibred category
\[ \mathbf{Coh}_{j^\dagger } \rightarrow \mathbf{Frame}_\mathcal{V}.\] We will explain this example in more detail in \S\ref{sec: frj} below.
\end{enumerate}
\end{example}

\section{Flat descent for analytic spaces}

The purpose of this section is to give a careful discussion of flat descent for analytic spaces, and in particular rephrasing the results of \cite{Con06} in terms of adic spaces. One particularly pleasing aspect of this reformulation is that it gives a very natural interpretation of the condition appearing in \cite[\S4]{Con06} that a flat map of rigid spaces `admits local fpqc sections' - it simply means that the induced map on adic spaces is surjective. This will then let us deduce a simple-to-state version of flat descent for analytic spaces over $K$.

First of all, we recall the notion of flatness from \cite[Chapitre 0, \S6.7]{EGA1}.

\begin{definition} A morphism $f:X\rightarrow Y$ of analytic spaces over $K$ is said to be:
\begin{enumerate} \item \emph{flat} if for all $x\in X$ the ring homomorphism $\mathcal{O}_{Y,f(x)}\rightarrow \mathcal{O}_{X,x}$ is flat.
\item \emph{faithfully flat} if it is flat and surjective.
\item \emph{fpqc} if it is faithfully flat and quasi-compact.
\end{enumerate}
\end{definition}

Note that the second condition is stronger than simply requiring surjectivity on rigid points, as the following example shows.

\begin{example} Let $X$ be the disjoint union of the open unit disc and the closed annulus of radius 1. Let $Y$ be the closed unit disc, and $f:X\rightarrow Y$ the obvious map. Then $f$ is flat and surjective on rigid points, but not faithfully flat in the above sense.

Indeed, by definition $Y=\spa{K\tate{z},\mathcal{V}\tate{z}}$ consists of $\varpi$-adically continuous valuations $v$ on $K\tate{z}$ satisfying $v(\mathcal{V}\tate{z})\leq 1$. Similarly, $X$ is the disjoint union of two subset of $Y$: the first consisting of those valuations $v$ that satisfy $v(z^n)\leq v(\varpi)$ for some $n\geq 1$ (the open disc of radius $1$), and the second consisting of those valuations satisfying $v(z)=1$ (the closed annulus of radius $1$).

The point is that there exists a valuation $v\in Y$ such that $v(z)<1$ but $v(z^n)>v(\varpi)$ for all $n$: namely, we take as value group the product $ \R_{>0} \times \gamma^\Z$, ordered so that $\rho<\gamma<1$ for all $\rho<1$, and we define
\begin{align*} v: K\tate{z} &\rightarrow \{0\} \cup  \left( \R_{>0} \times \gamma^\Z \right) \\
v\left(\sum_i a_iz^i \right) &= \sup_i \norm{a_i}\gamma^i
\end{align*}
where $\norm{\,\cdot\,}$ is any $p$-adic norm on $K$. Hence $f:X\rightarrow Y$ is not surjective.
\end{example}

Since we will be comparing with the situation of rigid spaces, let us recall the following definitions.

\begin{definition} A morphism $f_0:X_0\rightarrow Y_0$ of rigid spaces over $K$ is said to be:
\begin{enumerate}
\item \emph{flat} if for all $x\in X_0$ the ring homomorphism $\mathcal{O}_{Y_0,f_0(x)}\rightarrow \mathcal{O}_{X_0,x}$ is flat.
\item \emph{fpqc} if it is flat, quasi-compact and surjective.
\end{enumerate}
\end{definition}

We have deliberately avoided giving the definition of a faithfully flat map of rigid spaces without additional quasi-compactness hypotheses. We will first need to check various compatibilities of these notions. Note that it follows immediately from the definitions that a morphism $f_0:X_0\rightarrow Y_0$ of rigid spaces over $K$ is flat if the associated morphism $f:X\rightarrow Y$ of analytic spaces over $K$ is so, and in fact the converse is also true.

\begin{proposition} \label{cor: three} Let $f_0:X_0\rightarrow Y_0$ be a morphism of qcqs rigid spaces over $K$, with induced morphism $f:X\rightarrow Y$ of analytic spaces over $K$. Then the following are equivalent.
\begin{enumerate} \item $f_0$ is flat (resp. fpqc);
\item $f$ is flat (resp. fpqc);
\item there exists a flat (resp. fpqc) morphism $ \mathfrak{f}:\mathfrak{X}\rightarrow \mathfrak{Y}$ of admissible formal schemes over $\mathcal{V}$ whose induced morphism on rigid generic fibres is $f_0$, and on adic generic fibres is $f$.
\end{enumerate}
\end{proposition}

\begin{proof} We clearly have (2)$\Rightarrow$(1), let us show that (3)$\Rightarrow$(2). In the flat case this is relatively straightforward. Flatness is a local property, and on the level of formal schemes is equivalent to locally being of the form $\spf{B^\circ}\rightarrow \spf{A^\circ}$ with $A^\circ\rightarrow B^\circ$ a flat morphism of topologically finite type $\varpi$-adic $\mathcal{V}$-algebras. It therefore suffices to show that if $A\rightarrow B$ is a flat morphism of affinoid $K$-algebras, then $\spa{B,B^+}\rightarrow \spa{A,A^+}$ is flat. But this simply follows from the fact that analytic localisations are flat morphisms.

In the fpqc case this is a little more involved, the key claim is that if $\mathfrak{X}\rightarrow \mathfrak{Y}$ is fpqc, then the induced map $X\rightarrow Y$ on adic generic fibres is surjective. Applying \cite[Chapter II, Theorem A.5.2]{FK18}, we will divide the map $X \rightarrow Y$ into two parts. First of all, let $\mathcal{C}$ denote the category of admissible blow-ups of $\mathfrak{Y}$, and $\mathcal{D}$ that of $\mathfrak{X}$. Since blowups commute with flat base change, it follows that if $\mathcal{I}\subset \mathcal{O}_\mathfrak{Y}$ is an admissible ideal, and $\mathfrak{Y}'\rightarrow \mathfrak{Y}$ is the blowup of $\mathfrak{Y}$ along $\mathcal{I}$, then $\mathfrak{Y}'\times_\mathfrak{Y} \mathfrak{X}$ is the blowup of $\mathfrak{X}$ along the (necessarily admissible) ideal $\mathcal{I}\mathcal{O}_\mathfrak{X}$. There is therefore a canonical functor $\mathcal{C}\rightarrow \mathcal{D}$ taking $\mathfrak{Y}'\rightarrow \mathfrak{Y}$ to $\mathfrak{X}':=\mathfrak{Y}'\times_\mathfrak{Y} \mathfrak{X} \rightarrow \mathfrak{X}$, note that the induced map $\mathfrak{X}'\rightarrow \mathfrak{Y}'$ is therefore faithfully flat. We now consider the maps
\[ X =\varprojlim_{ \mathfrak{X}'\rightarrow\mathfrak{X}\in \mathcal{D}} \mathfrak{X}' \rightarrow \varprojlim_{\mathfrak{Y}'\rightarrow \mathfrak{Y}\in \mathcal{C}} \mathfrak{X}' \rightarrow \varprojlim_{\mathfrak{Y}'\rightarrow \mathfrak{Y}\in \mathcal{C}} \mathfrak{Y}' =Y, \]
and can conclude by applying \cite[Theorem 0.2.2.13]{FK18} twice.

Finally, we note that (1)$\Rightarrow$(3) in the flat case follows from \cite[Theorem 7.1]{Bos09}, for the fpqc case, we can argue as follows. Suppose we have a flat formal model $\mathfrak{f}:\mathfrak{X}\rightarrow \mathfrak{Y}$ of $X_0\rightarrow Y_0$ which is \emph{not} fpqc. Then arguing as in \cite[Coroallry 7.2]{Bos09} we see that the image of $\mathfrak{f}$ is a proper open subset $\mathfrak{U}$ of $\mathfrak{Y}$, and we get a factorisation $\mathfrak{X}\rightarrow \mathfrak{U} \rightarrow \mathfrak{Y}$. Thus we obtain a factorisation $X_0\rightarrow U_0\rightarrow Y_0$ where $U_0\subset Y_0$ is a proper open sub-variety, contradicting surjectivity of $f_0$.
\end{proof}

\begin{corollary} \label{cor: open} Let $f:X\rightarrow Y$ be a flat morphism of analytic spaces over $K$. Then $f$ is open.
\end{corollary}

\begin{proof} The question is local on both $Y$ and $X$, we may therefore assume them both to be qcqs. It moreover suffices to show that the image $f(X)$ is open. We know that there exists a flat formal model $\mathfrak{X}\rightarrow \mathfrak{Y}$ of $f$, and as above we can see that this map has to factor as a fpqc map $\mathfrak{X}\rightarrow \mathfrak{U}$ followed by an open immersion $\mathfrak{U}\rightarrow \mathfrak{Y}$. We can now apply Proposition \ref{cor: three} above.
\end{proof}

Proposition \ref{cor: three} begs the question of what the `rigid' analogue of faithful flatness is, and rather pleasingly this turns out to be exactly the descent condition appearing in \cite[Theorem 4.2.8]{Con06}.

\begin{definition} A flat map $f_0:X_0\rightarrow Y_0$ of rigid spaces is said to `admit local fpqc sections' if there exists an admissible cover $Y_0=\bigcup_i Y_{0,i}$ of $Y_0$, fpqc maps $Z_{0,i}\rightarrow Y_{0,i}$, and for each $i$ factorisations $Z_{0,i} \rightarrow X_0\rightarrow Y_0$ of $Z_{0,i}\rightarrow Y_0$.
\end{definition}

The proof of the following is similar to that of \cite[Theorem 4.2.8]{Con06}.

\begin{theorem} Let $f_0:X_0\rightarrow Y_0$ be a flat morphism of rigid spaces over $K$, with $f:X\rightarrow Y$ the induced morphism of analytic spaces over $K$. Then $f_0$ admits local fpqc sections if and only if $f$ is faithfully flat.
\end{theorem}

\begin{proof} First suppose that $f_0$ admits fpqc local sections, we must show that $f$ is surjective. This is clearly local for an admissible covering of $Y_0$, hence we may assume that $Y_0$ is affinoid, and that there exists an fpqc map $Z_0\rightarrow Y_0$ and a commutative diagram
\[ \xymatrix{ & X_0 \ar[d] \\ Z_0 \ar[r]\ar[ur] & Y_0.  }  \]
If we let $Z$ denote the analytic space corresponding to $Z_0$, it then suffices to show that $Z\rightarrow Y$ is surjective, which follows from Proposition \ref{cor: three} above.

Conversely, let us suppose that $f$ is faithfully flat, we wish to show that $f_0$ admits local fpqc sections. This question is clearly local for an admissible cover of $Y_0$, which we may therefore assume to be affinoid, and in particular qcqs. Thus for any qcqs open $V\subset X$ we know that $f(V)$ is a qcqs open in $Y$ by Corollary \ref{cor: open} above. Hence the fpqc map $V\rightarrow f(V)$ between qcqs analytic spaces over $K$ has to come from an fpqc map $V_0\rightarrow f_0(V_0)$ of Tate spaces over $K$. As $V$ ranges over an open cover of $X$ by qcqs opens, the images $f(V)$ form an open cover of $Y$. Hence the images $f_0(V_0)$ form an admissible open cover of $Y_0$, and $f_0$ admits fpqc local sections.
\end{proof}

By following the proof of this theorem, that is arguing as in the proof of \cite[Theorem 4.2.8]{Con06}, we can deduce the required descent result from \cite[Theorem 3.1]{BG98}.

\begin{theorem} \label{theo: cohdes} Let $f:X\rightarrow Y$ be a faithfully flat morphism of analytic spaces over $K$. Then $f$ is a morphism of effective descent for coherent sheaves.
\end{theorem}

\begin{proof} We may assume that $Y$ is affinoid. Let $\{V_i\}$ be an open cover of $X$ by qcqs opens, and let $U_i=f(V_i)$. Then we have a commutative diagram
\[\xymatrix{\CMcoprod_i V_i \ar[r]\ar[d] & X \ar[d] \\ \CMcoprod_iU_i  \ar[r] & Y } \]
and since we know effective descent for open covers, it suffices to show that each $f_i:V_i\rightarrow U_i$ is of effective descent for coherent sheaves. But now $f_i$ is an fpqc morphism between qcqs analytic spaces, in particular it comes from an fpqc morphism $V_{0,i}\rightarrow U_{0,i}$ of rigid spaces. Hence we may apply \cite[Theorem 3.1]{BG98} to conclude. 
\end{proof}

\begin{remark} Of course, \cite[Theorem 3.1]{BG98} applies to coherent module on rigid spaces, not analytic spaces. However, if $X_0$ is a qcqs rigid space, with associated analytic space $X$, then there is a canonical equivalence of categories $\mathrm{Coh}(\mathcal{O}_{X_0}) \cong \mathrm{Coh}(\mathcal{O}_X)$, functorial in $X$.  
\end{remark}

\section{Tubes and isocrystals in adic geometry}\label{sec: frj}

In this section, we will give a quick review of Berthelot's theorem of tubes and overconvergence, from the point of view of adic geometry.

\begin{definition} A frame $(X,Y,\mathfrak{P})$ over $\mathcal{V}$ is a triple consisting of an open immersion $X\hookrightarrow Y$ of $k$-varieties and a closed immersion $Y\hookrightarrow \mathfrak{P}$ of formal $\mathcal{V}$-schemes.
\end{definition}

Let $\mathfrak{P}_K$ denote the adic generic fibre of $\mathfrak{P}$. Then we have a \emph{continuous} specialisation map
\[ \mathrm{sp}:\mathfrak{P}_K\rightarrow \mathfrak{P}_k \]
and we define $]Y[_\mathfrak{P}:=\mathrm{sp}^{-1}(Y)^{\circ}$ to be the topological interior of the inverse image of $Y$ under $\mathrm{sp}$. Thus we have a continuous specialisation map
\[ \mathrm{sp}_Y:]Y[_\mathfrak{P} \rightarrow Y \]
and we define $]X[_\mathfrak{P}:=\overline{\mathrm{sp}_Y^{-1}(X)}$ to be the topological closure of the inverse image of $X$ under $\mathrm{sp}_Y$. This only depends on the embedding $X\hookrightarrow \mathfrak{P}$ and not on $Y$, since it can be seen as the inverse image of the tube of $X$ inside the Berkovich generic fibre of $\mathfrak{P}$, via the separation map from the adic to the Berkovich generic fibre. We have a closed immersion
\[ j:]X[_\mathfrak{P} \rightarrow ]Y[_\mathfrak{P} \]
and we define, for any sheaf $\mathcal{F}$ on $]Y[_\mathfrak{P}$, the sheaf $j_X^\dagger\mathcal{F}:=j_*j^{-1}\mathcal{F}$ on $]Y[_\mathfrak{P}$. If $V\subset ]Y[_\mathfrak{P}$ is a quasi-compact open, then
\[ \Gamma(V,j_X^\dagger\mathcal{O}_{]Y[_\mathfrak{P}}) = \mathrm{colim}_{W\supset V\cap ]X[_\mathfrak{P}}\Gamma(W,\mathcal{O}_W) ,\]
is a filtered colimit of Noetherian rings with flat transition maps, hence coherent,\footnote{in fact, it is Noetherian, but this is a little harder to show} and so $j_X^\dagger\mathcal{O}_{]Y[_\mathfrak{P}}$ is a coherent sheaf of rings. In particular, we may consider the category $\mathrm{Coh}(j_X^\dagger\mathcal{O}_{]Y[_\mathfrak{P}})$ of coherent $j_X^\dagger\mathcal{O}_{]Y[_\mathfrak{P}}$-modules on $]Y[_\mathfrak{P}$, and by \cite[Proposition 2.20]{LP16} this category is canonically equivalent to its rigid analogue as considered in \cite[\S5.4]{LS07}.

If $f:(W,Z,\mathfrak{Q})\rightarrow (X,Y,\mathfrak{P})$ is a morphism of frames there is an obvious pull-back functor
\[ f^\dagger: \mathrm{Coh}(j_X^\dagger\mathcal{O}_{]Y[_\mathfrak{P}})\rightarrow \mathrm{Coh}(j_W^\dagger\mathcal{O}_{]Z[_\mathfrak{Q}}) \]
which again can be identified with its rigid analogue by the functoriality of the identification in \cite[Proposition 2.20]{LP16}. In particular, we may apply \cite[\S8]{LS07} to obtain the following description of the category $\isoc{X/K}$ of overconvergent isocrystals. 

\begin{theorem} The category $\isoc{X/K}$ is canonically equivalent to the following category.
\begin{itemize} \item Objects:
\begin{itemize} \item for every frame $(T,\overline{T},\mathfrak{T})$ equipped with a map $T\rightarrow X$, a coherent $j_T^\dagger\mathcal{O}_{]\overline{T}[_\mathfrak{T}}$-module $E_T$;
\item for every map $g:(T',\overline{T}',\mathfrak{T}')\rightarrow (T,\overline{T},\mathfrak{T})$ over $X$ an isomorphism $u_g:g^\dagger E_T\rightarrow E_{T'}$;
\item such that $u_{g'}\circ g'^\dagger(u_g)=u_{g\circ g'}$ for composable morphisms of frames $g,g'$ over $X$.
\end{itemize}
\item Morphisms:
\begin{itemize}
\item for every frame $(T,\overline{T},\mathfrak{T})$ equipped with a map $T\rightarrow X$, a morphism $\psi_T:E_{1,T}\rightarrow E_{2,T}$;
\item such that for every morphism of frames $g:(T',\overline{T}',\mathfrak{T}')\rightarrow (T,\overline{T},\mathfrak{T})$ over $X$, the diagram
\[\xymatrix{  g^\dagger E_{1,T} \ar[r]^{g^\dagger\psi_T} \ar[d]_{u_g} & g^\dagger E_{2,T} \ar[d]^{u_g} \\ E_{1,T'} \ar[r]^{\psi_{T'}} & E_{2,T'}  }  \]
commutes.
\end{itemize}
\end{itemize}
\end{theorem}

\section{Descent for coherent \texorpdfstring{$j^\dagger$}{j}-modules}

The strategy of proof of Theorem \ref{theo: main} below will be to follow that of Ogus \cite[Theorem 4.6]{Ogu84} in the convergent case, and just as the key component of the proof there is a version of flat descent for coherent sheaves on rigid spaces, so we will need a version of flat descent for coherent $j^\dagger$-modules on frames.

\begin{theorem} \label{theo: jdes} Let $f:(X,\overline{X},\mathfrak{X})\rightarrow (T,\overline{T},\mathfrak{T})$ be a morphism of frames such that:
\begin{enumerate}
\item $X\rightarrow T$ is proper surjective;
\item $\overline{X}\rightarrow \overline{T}$ is proper;
\item $\mathfrak{X} \rightarrow \mathfrak{T}$ is flat. 
\end{enumerate}
Then $f$ is a morphism of effective descent for coherent $j^\dagger$-modules.
\end{theorem}

Unsurprisingly, the idea will be to reduce to flat descent for rigid analytic varieties, and the key lemma that will enable us to do so is the following.

\begin{lemma} \label{lemma: product} Let $f:(X,\overline{X},\mathfrak{X})\rightarrow (T,\overline{T},\mathfrak{T})$ be a morphism of frames. Then for every neighbourhood
\[ ]X\times_T X[_{\mathfrak{X}\times_\mathfrak{T}\mathfrak{X}} \subset W \subset ]\overline{X}\times_{\overline{T}}\overline{X}[_{\mathfrak{X}\times_\mathfrak{T}\mathfrak{X}} \]
of $]X\times_T X[_{\mathfrak{X}\times_\mathfrak{T}\mathfrak{X}}$ in $]\overline{X}\times_{\overline{T}}\overline{X}[_{\mathfrak{X}\times_\mathfrak{T}\mathfrak{X}}$ there exists a neighbourhood
\[ ]X[_\mathfrak{X} \subset V \subset ]\overline{X}[_{\mathfrak{X}}\]
of $]X[_\mathfrak{X}$ in $]\overline{X}[_{\mathfrak{X}}$ such that $V\times_{\mathfrak{T}_K} V \subset W$. 
\end{lemma}

\begin{proof} This is similar to \cite[Proposition 3.2.12]{LS07}. The question is local on both $\mathfrak{T}$ and $\mathfrak{X}$, hence we may assume that they are both affine. In particular, we may choose functions $f_1,\ldots,f_r,g_1,\ldots,g_s\in \Gamma(\mathfrak{X},\mathcal{O}_\mathfrak{X})$ such that $\overline{X}= \bigcap V(f_i)\subset \mathfrak{X}_k$ and $X=\bigcup_j D(g_j)\subset \overline{X}$. Hence we have
\[ \overline{X}\times_{\overline{T}}\overline{X}= \bigcap_{i} V(f_i\otimes 1,1\otimes f_i) \subset \mathfrak{X}_k\times_{\mathfrak{T}_k} \mathfrak{X}_k\]
and
\[ X\times_T X = \bigcup_{j,j'} D(g_j\otimes g_{j'}) \subset  \overline{X}\times_{\overline{T}}\overline{X}.  \]
By \cite[\S1.2.4]{Ber96b} we may assume that there exists an increasing sequence $m_n$ of integers such that
\[ W=\bigcup_{n} \left\{ \left. x\in \mathfrak{X}_K\times_{\mathfrak{T}_K}  \mathfrak{X}_K \right\vert v_x(\varpi^{-1}f^n_i \otimes 1), v_x(\varpi^{-1} \otimes f^n_i) \leq 1\;\forall i,i',\;\;\exists j,j'\text{ s.t. }v_x(\varpi^{-1}g^{m_n}_{j}\otimes g_{j'}^{m_n}) \geq 1  \right\}.  \]
Again applying \cite[\S1.2.4]{Ber96b} we may construct the neighbourhood
\[ V=\bigcup_{n} \left\{ \left. x\in \mathfrak{X}_K \right\vert v_x(\varpi^{-1}f^n_i)\leq 1\;\forall i,i',\;\;\exists j\text{ s.t. }v_x(\varpi^{-1}g^{2m_n}_{j}) \geq 1  \right\}  \]
of $]X[_{\mathfrak{X}}$ inside $]\overline{X}[_\mathfrak{X}$ which clearly satisfies $V\times_{\mathfrak{T}_K}V \subset W$.
\end{proof}

We can now prove Theorem \ref{theo: jdes}.

\begin{proof}[Proof of Theorem \ref{theo: jdes}]
We may assume that $X$ and $T$ are dense in $\overline{X}$ and $\overline{T}$ respectively, from which we deduce that $\overline{X}\rightarrow \overline{T}$ is also proper and surjective. In particular, we can see that the square
\[ \xymatrix{ X \ar[r]\ar[d] & \overline{X}\ar[d] \\ T\ar[r] &\overline{T}  } \]
is Cartesian. In particular, by functoriality of the separation map from the adic generic fibre of $\mathfrak{X}$ to its Berkovich generic fibre, we have $f(]\overline{X}\setminus X[_{\mathfrak{X}})=]\overline{T}\setminus T[_\mathfrak{T}$. Hence using Lemma \ref{cor: open} we can deduce that if
\[ V\subset ]\overline{X}[_\mathfrak{X} \]
is a neighbourhood of $]X[_\mathfrak{X}$, then 
\[ f(V)\subset  ]\overline{T}[_\mathfrak{T}\]
must be a neighbourhood of $]T[_\mathfrak{T}$. Moreover, if $\{V\}$ forms a cofinal system of neighbourhoods of $]X[_\mathfrak{X}$ in $]\overline{X}[_\mathfrak{X}$, then $\{f(V)\}$ forms a  cofinal system of neighbourhoods of $]T[_\mathfrak{T}$ in $]\overline{T}[_\mathfrak{T}$.

In particular, for any such $V$ we may consider the category
\[ \mathrm{Coh}(V\times_{f(V)}V \rightrightarrows V) \]
of coherent $\mathcal{O}_V$-modules together with descent data relative to $V\rightarrow f(V)$. Since $V\times_{f(V)} V$ is a neighbourhood of $]X\times_T X[_{\mathfrak{X}\times_\mathfrak{T}\mathfrak{X}}$ in $]\overline{X}\times_{\overline{T}}\overline{X}[_{\mathfrak{X}\times_\mathfrak{T}\mathfrak{X}}$ we therefore obtain a pull-back functor
\[ \mathrm{Coh}(V\times_{f(V)}V \rightrightarrows V) \rightarrow  \mathrm{Coh}(j_X^\dagger\mathcal{O}_{]\overline{X}[_\mathfrak{X}} \rightrightarrows j^\dagger_{X\times_T X} \mathcal{O}_{]\overline{X}\times_{\overline{T}} \overline{X}[_{\mathfrak{X}\times_\mathfrak{T}\mathfrak{X}}} )  \]
and hence a functor
\[ 2\text{-}\mathrm{colim}_V \mathrm{Coh}(V\times_{f(V)}V \rightrightarrows V) \rightarrow  \mathrm{Coh}(j_X^\dagger\mathcal{O}_{]\overline{X}[_\mathfrak{X}} \rightrightarrows j^\dagger_{X\times_T X} \mathcal{O}_{]\overline{X}\times_{\overline{T}} \overline{X}[_{\mathfrak{X}\times_\mathfrak{T}\mathfrak{X}}} ). \]
It follows from Lemma \ref{lemma: product} together with \cite[Theorem 5.4.4]{LS07} that this functor is an equivalence of categories. By Theorem \ref{theo: cohdes} above we have an equivalence of categories
\[\mathrm{Coh}(V\times_{f(V)}V \rightrightarrows V ) \cong \mathrm{Coh}(f(V)),\]
and hence once more applying \cite[Theorem 5.4.4]{LS07} finishes the proof.
\end{proof}

\section{Effective descent for isocrystals}

We can now give the proof of the first main descent result.

\begin{theorem} \label{theo: main} Let $f:X\rightarrow Z$ be a proper surjective map of $k$-varieties. Then $f$ is a morphism of effective descent for overconvergent isocrystals.
\end{theorem}

Throughout the proof, we will use Le Stum's `site-theoretic' characterisation of overconvergent isocrystals \cite[\S8]{LS07} recalled above.

\begin{proof} By Chow's lemma we can find a projective, surjective map $X'\rightarrow X$ such that the composite $X'\rightarrow Z$ is also projective. By the general descent formalism, it suffices to prove that both $X'\rightarrow Z$ and $X'\rightarrow X$ are morphisms of effective descent for overconvergent isocrystals, in other words we may assume that $f$ is projective, and since the question is also local on $Z$, we may assume that we have some closed immersion $X\hookrightarrow \P^n_Z$. The basic construction we will use is the following.

Let $(T,\overline{T},\mathfrak{T})$ be a frame equipped with a map to $Z$. Then we may base change $X\hookrightarrow \P^n_Z$ by $T\rightarrow Z$ to obtain
\[ X_T \hookrightarrow \P^n_T \]
and hence we may extend $X_T\rightarrow T$ to a morphism of frames
\[ \xymatrix{ X_T \ar[r]\ar[d] & \overline{X}_T \ar[r]\ar[d] & \widehat{\P}^n_\mathfrak{T} \ar[d] \\ T\ar[r] & \overline{T} \ar[r]& \mathfrak{T} }\]
where $\overline{X}_T$ is simply the closure of $X_T$ inside $\P^n_{\overline{T}}$. Since proper surjective maps are stable by base change, this morphism satisfies the conditions of Theorem \ref{theo: jdes}.

Let us first consider the `fully faithful' part of descent. First of all, the functor $f^*:\isoc{Z/K}\rightarrow \isoc{X/K}$ is faithful by \cite[Lemma 1.8]{Cre92}, since every geometric component of $X$ is mapped to by a geometric component of $Z$. The functor
\[ f^*: \isoc{Z/K}\rightarrow \isoc{X\times_Z X \rightrightarrows X/K } \]
is therefore faithful, and it suffices to show fullness. So suppose that we have objects $E_1,E_2\in \isoc{Z/K}$ and a morphism $\psi:f^*E_1\rightarrow f^*E_2$ compatible with the canonical descent data. Then for any frame $(T,\overline{T},\mathfrak{T})$ equipped with a map $T\rightarrow Z$ we may form $f_T:(X_T,\overline{X}_T,\widehat{\P}^n_\mathfrak{T})\rightarrow (T,\overline{T},\mathfrak{T})$ as above, and realise $E_i$ on $(T,\overline{T},\mathfrak{T})$ to obtain coherent $j^\dagger\mathcal{O}_{]\overline{T}[_\mathfrak{T}}$-modules $E_{i,T}$. Moreover, we may realise $\psi$ on $(X_T,\overline{X}_T,\widehat{\P}^n_\mathfrak{T})$ to obtain a morphism
\[ \psi_T: f_T^\dagger E_{1,T} \rightarrow f_T^\dagger E_{2,T},\]
which has to be compatible with the natural descent data on $E_{i,T}$ relative to $f_T:(X_T,\overline{X}_T,\widehat{\P}^n_\mathfrak{T})\rightarrow (T,\overline{T},\mathfrak{T})$. Hence by Theorem \ref{theo: jdes} it has to come from a unique morphism $E_{1,T} \rightarrow E_{2,T}$ of coherent $j^\dagger\mathcal{O}_{]\overline{T}[_\mathfrak{T}}$-modules. By uniqueness these are then compatible as $(T,\overline{T},\mathfrak{T})$ varies, and hence give rise to a morphism $E_1\rightarrow E_2$ of overconvergent isocrystals on $Z/K$.

Next, let us treat the effectivity part of descent. So let $E\in\isoc{X/K}$ be equipped with descent data relative to $f$; we wish to produce an overconvergent isocrystal on $Z/K$, and we will do so by constructing its realisations on any frame $(T,\overline{T},\mathfrak{T})$ equipped with a map $T\rightarrow Z$. In this situation, we may form $(X_T,\overline{X}_T,\widehat{\P}^n_\mathfrak{T})\rightarrow (T,\overline{T},\mathfrak{T})$ as above, and realise $E$ on $(X_T,\overline{X}_T,\widehat{\P}^n_\mathfrak{T})$ to obtain a coherent $j_{X_T}^\dagger\mathcal{O}_{]\overline{X}_T[_{\widehat{\P}^n_\mathfrak{T}}}$-module $E_{X_T}$. Moreover, the descent data for $E$ relative to $X\rightarrow Z$ gives rise to descent data for $E_{X_T}$ relative to
\[ (X_T,\overline{X}_T,\widehat{\P}^n_\mathfrak{T})\rightarrow (T,\overline{T},\mathfrak{T}).\]
Hence by Theorem \ref{theo: jdes} we obtain a coherent $j_T^\dagger\mathcal{O}_{]\overline{T}[_\mathfrak{T}}$-module $F_T$ whose pullback to $(X_T,\overline{X}_T,\widehat{\P}^n_\mathfrak{T})$ is $E_{X_T}$. Note that once our original embedding $X\rightarrow \P^n_Z$ was fixed the construction of $F_T$ is completely canonical, and does not depend on any further choices. Thus one easily checks using the corresponding properties of $E$ together with Theorem \ref{theo: jdes} that if $g: (T',\overline{T}',\mathfrak{T}')\rightarrow  (T,\overline{T},\mathfrak{T})$ is a morphism of frames over $Z$, then there is a corresponding isomorphism $g^\dagger F_{T}\isomto F_{T'}$, and these moreover satisfy the cocycle condition. Hence there is a unique overconvergent isocrystal $F$ on $Z/K$ whose realisation on each $(T,\overline{T},\mathfrak{T})$ is exactly $F_T$. This completes the proof in the proper surjective case.
 \end{proof}
 
 Recall that the h topology on varieties over $k$ is the coarsest topology for which both Zariski open covers and proper, surjective maps are covering maps, in particular it is finer than the fppf topology \cite{Voe96}.
 
 \begin{corollary} \label{cor: main2} Any h cover (and hence, an fppf cover) of $k$-varieties is a morphism of effective descent for overconvergent isocrystals.
 \end{corollary}

\section{Topological invariance and equivalence of Frobenius pull-back}

The main application of Theorem \ref{theo: main} we have in mind is the following.

\begin{theorem} \label{theo: topinv}
Let $f:X\rightarrow Z$ be a universal homeomorphism (i.e. $f$ is finite, surjective and radicial). Then 
\[ f^*:\isoc{Z/K}\rightarrow \isoc{X/K} \]
is an equivalence of categories.
\end{theorem}

\begin{proof}
This is a formal consequence of h descent, we give a sketch of the (well-known) argument. Any such map $f$ is an h cover, so it suffices to show that the pullback functor
\[\isoc{X/K} \rightarrow \isoc{X\times_Z X\rightrightarrows X/K} \]
is an equivalence of categories. If $f$ is a nilpotent closed immersion, then $X\times_Z X\isomto X$ via either projection, and the claim is therefore clear. In general, the diagonal $X\hookrightarrow X\times_Z X$ is a nilpotent closed immersion, so by the previously considered case the pullback $\isoc{X\times_Z X/K} \rightarrow \isoc{X/K}$ is an equivalence, and again the claim follows.
\end{proof}

Finally, let us suppose that we have chosen a lift $\sigma$ to $K$ of the $q$-power Frobenius on $k$. Thus we obtain a $q$-power Frobenius pull-back functor
\[ F^*:\isoc{X/K}\rightarrow \isoc{X/K}. \]

\begin{corollary} \label{cor: Fequiv}
Assume that $k$ is perfect, and let $X$ be any $k$-variety. Then $F^*$ is an equivalence of categories.
\end{corollary}

\begin{proof} Let $X^{(q)}$ be the pull-back of $X$ by the $q$-power Frobenius of $k$. Since $k$ is perfect, the corresponding (semi-linear) pull-back functor
\[ \isoc{X/K}\rightarrow \isoc{X^{(q)}/K} \]
is an equivalence of categories. It therefore suffices to observe that the relative Frobenius $F_{X/k}:X\rightarrow X^{(q)}$ is a universal homeomorphism and apply Theorem \ref{theo: topinv}.
\end{proof}

\bibliographystyle{../../Templates/bibsty}
\bibliography{../../lib.bib}

\end{document}